\documentclass[11pt]{article}
\usepackage[margin=1.5in]{geometry}

\usepackage[english]{babel}
\usepackage{relsize}
\usepackage{verbatim}
\usepackage{amsmath,amssymb,amsthm}

\usepackage[maxnames=50,backend=bibtex,firstinits=true]{biblatex}

\usepackage{array}
\newcolumntype{C}[1]{>{\arraybackslash$}p{#1}<{$}}

\usepackage[noend,noline]{algorithm2e}

\usepackage{ifthen}
\usepackage{tikz}
\usetikzlibrary{calc}

\DeclareMathOperator{\xc}{xc}

\newcommand{\R}{\mathbb R}
\newcommand{\N}{\mathbb N}

\newcommand{\conv}{\mathrm{conv}}

\newcommand{\calR}{\mathcal{R}}
\newcommand{\calS}{\mathcal{S}}

\newcommand{\stab}{\mathrm{STAB}}
\newcommand{\qstab}{\mathrm{QSTAB}}

\newcommand{\bmat}[1]{\begin{bmatrix}#1\end{bmatrix}}

\newcommand{\logn}{\lceil\log_2 n\rceil}

\newtheorem{prop}{Proposition}
\newtheorem{lem}[prop]{Lemma}
\newtheorem{obs}[prop]{Observation}
\newtheorem{thm}[prop]{Theorem}

\theoremstyle{definition}

\begin{document}

\pagenumbering{arabic}

\title{Extended formulations from communication\\ protocols in output-efficient time}
%
%

\title{Extended formulations from communication\\ protocols in output-efficient time}
\author{Manuel Aprile, Yuri Faenza}
\date{}
\maketitle          
\begin{abstract}
  Deterministic protocols are well-known tools to obtain extended formulations, with many applications to polytopes arising in combinatorial optimization. Although constructive, those tools are not output-efficient, since the time needed to produce the extended formulation also depends on the number of rows of the slack matrix (hence, on the exact description in the original space). We give general sufficient conditions under which those tools can be implemented as to be output-efficient, showing applications to e.g.~Yannakakis' extended formulation for the stable set polytope of perfect graphs, for which, to the best of our knowledge, an efficient construction was previously not known. For specific classes of polytopes, we give also a direct, efficient construction of extended formulations arising from protocols. Finally, we deal with extended formulations coming from unambiguous non-deterministic protocols.
\smallskip

\end{abstract}
\section{Introduction}\label{sec:intro}

Linear extended formulations are a fundamental tool in integer programming and combinatorial optimization, since they allow to reduce an optimization problem over a polyhedron $P$ to an analogous one over a polyhedron $Q$ that linearly projects to $P$. When $Q$ can be described with much fewer inequalities than $P$ (typically, polynomial vs. exponential in the dimension of $P$), this leads to a computational speedup. $Q$ as above is called an \emph{extension} of $P$, any set of linear inequalities describing $Q$ is an \emph{extended formulation}, and the minimum number of inequalities in an extended formulation for $P$ is called the \emph{extension complexity} of $P$, and denoted by $\xc(P)$. Computing or bounding the extension complexity of polytopes has been an important topic in recent years, see e.g.~\cite{chan2016approximate,fiorini2012linear,rothvoss2017matching}. 

Lower bounds on extension complexity are usually unconditional: neither they rely on any complexity theory assumptions, nor they take into account the time needed to produce the extension or the encoding length of coefficients in the inequalities. Upper bounds are often constructive and produce an extended formulation in time polynomial (often linear) in its size. Examples of the latter include Balas' union of polytopes, reflection relations, and branched polyhedral branching systems (see e.g.~\cite{conforti2010extended,kaibel2011extended}) .

The fact that we can construct extended formulations \emph{efficiently} is crucial, since their final goal is to make certain optimization problems (more) tractable. It is interesting to observe that there is indeed a gap between the \emph{existence} of certain extended formulations, and the fact that we can construct them efficiently: for instance in~\cite{bazzi2018no}, it is shown that there is a small extended formulation for the stable set polytope that is $O(\sqrt{n})$-approximated (with $n$ being the number of nodes of the graph), but we do not expect to obtain it efficiently because of known hardness results~\cite{haastad2001some}. In another case, a proof of the existence of a subexponential formulation with integrality gap $2+\epsilon$ for min-knapsack~\cite{bazzi2017small} predated the efficient construction of a formulation with these properties~\cite{fiorini1711strengthening}. 

\smallskip

In this paper, we investigate the efficiency of an important tool for producing extended formulation: communication protocols. In a striking result, Yannakakis~\cite{yannakakis1991expressing} showed that a deterministic communication protocol computing the slack matrix of a polytope $P=\{ x: Ax \leq b \} \subseteq \mathbb{R}^n$ can be used to produce an extended formulation for $P$. The number of inequalities of the latter is at most $2^c$, where $c$ is the \emph{complexity} of the protocol (see Section \ref{sec:prelim} for definitions). Hence, deterministic protocols can be used to provide upper bounds on extension complexity of polytopes. This reduction is constructive, but not efficient. Indeed, it produces an extended formulation with $n+2^c$ variables, $2^c$ inequalities, and an equation per row of $A$. Basic linear algebra implies that most equations are redundant, but in order to eliminate those we may have to go through the full (possibly exponential-size) list.
The main application of Yannakakis' technique is arguably given in his original paper, and it provides a relaxation of the stable set polytope. In particular, it gives an exact formulation for the stable set polytope when $G$ is a perfect graph. This is a class of polytopes that has received much attention in the literature, see e.g.~\cite{Chvatal75,grotschel1984polynomial}. They also play an important role in extension complexity: while many open problems in the area were settled one after the other~\cite{chan2016approximate,fiorini2012linear,kaibel2013constructing,rothvoss2017matching}, we still do not know if the stable set polytope of perfect graphs has polynomial-size extension complexity. Yannakakis' protocol gives an upper bound of $n^{O(\log n)}$, while the best lower bound is as small as $\Omega(n \log n )$~\cite{aprile2017extension}. On the other hand, a maximum stable set in a perfect graph can be computed efficiently via a polysize \emph{semidefinite} extension known as Lovasz' Theta body~\cite{lovasz1979shannon}. This can also be used, together with the ellipsoid method, to efficiently find a coloring of a perfect graph, see e.g.~\cite[Section 67.1]{schrijver2002combinatorial}. We remark that designing a combinatorial (or at least SDP-free) polynomial-time algorithm to find a maximum stable set in perfect graphs, or to color them, is a main open problem~\cite{chudnovsky2015coloring}. 

\smallskip

\noindent {\bf Our results.} In this paper, we investigate conditions under which we can explicitly obtain an extended formulation from a communication protocol in time polynomial in the size of the formulation itself. We first show a general algorithm that achieves this for any deterministic protocol, given a compact representation of the protocol as a labelled tree and of certain extended formulations associated to leaves of the protocol. The algorithm runs in time linear in the input size and is flexible, in that it also handles non-exact extended formulations. We then show that in some cases one can obtain those extended formulations directly, without relying on this general algorithm. This may be more interesting computationally. We show applications of our techniques in the context of (not only) perfect graphs. Our most interesting application is to Yannakakis' original protocol, and it leads to an extended formulation of size $n^{O(\log n)}$ that can be constructed in time $n^{O(n \log n)}$. For perfect graphs, this gives a subexponential SDP-free algorithm that computes a maximum stable set (resp. an optimal coloring). For general graphs, this gives a new relaxation of the stable set polytope which is (strictly) contained in the clique relaxation.
 Finally, we extend our result to obtain extended formulations from unambiguous non-deterministic protocols.

\smallskip

\section{Preliminaries} \label{sec:prelim}

{\bf Communication protocols.} We start by describing the general setting of communication protocols, referring to~\cite{kushilevitz1996communication} for more details. 
Let $M$ be a non-negative matrix with row (resp. column) set $X$ (resp. $Y$), and two agents Alice and Bob. Alice is given as input a row index $i\in X$, Bob a column index $j\in Y$, and they aim at determining $M_{ij}$ by exchanging information according to some pre-specified mechanism, that goes under the name of \emph{deterministic protocol}. At each step of the protocol, the actions of Alice (resp. Bob) only depend on her (resp. his) input and on what they exchanged so far. The protocol is said to \emph{compute $M$} if, for any input $i$ of Alice and $j$ of Bob, it returns $M_{ij}$. Such a protocol can be modelled as a rooted tree, with each vertex modelling a step where exactly one of Alice or Bob sends a bit (hence labelled with $A$ or $B$), and its children representing subsequent steps of the protocol. The \emph{leaves} of the tree indicate the termination of the protocol and are labelled with the corresponding output. The tree is therefore binary, with each edge representing a 0 or a 1 sent. Hence, a deterministic protocol can be identified by the following parameters: a rooted binary tree $\tau$ with node set ${\cal V}$;  a function $\ell:{\cal V}\rightarrow \{A,B\}$ (``Alice'',``Bob''), associating each vertex to its type; for each leaf $v \in {\cal V}$,  a non-negative number $\Lambda_v$ corresponding to the output at $v$; for each $v\in {\cal V}$, the set $S_v$ of pairs $(i,j)\subseteq X \times Y$ such that, on input $(i,j)$, the step corresponding to node $v$ is executed during the protocol. We represent this compactly by $(\tau,\ell,\Lambda, \{S_v\}_{v \in {\cal V}})$. It can be shown that each $S_v$ is a \emph{rectangle}, i.e. a submatrix of $M$. For a leaf $v$ of $\tau$, all entries of $S_v$ have the same value $\Lambda_v$, i.e. they form a submatrix of $M$ with constant values. Such submatrices are called \emph{monochromatic rectangles}, their collection is denoted by $\calR$, and we say that $\Lambda_R:=\Lambda_v$ is the \emph{value} of the rectangle $R$ associated to the leaf $v$.
We assume that $S_v \neq \emptyset$ for each node $v$ of the protocol.

An \emph{execution} of the protocol is a path of $\tau$ from the root to a leaf, whose edges correspond to the bits sent during the execution. 
The \emph{complexity} of the protocol is given by the height $h$ of the tree $\tau$. A deterministic protocol computing $M$ gives a partition of $M$ in at most $2^h$ monochromatic rectangles.  We remark that one can obtain a protocol (and a partition in rectangles) for $M^T$ given a protocol for $M$ by just exchanging the roles of Alice and Bob. 

The setting of \emph{non-deterministic} protocols is similar as before, but now Alice and Bob are allowed to make guesses in their communication, with the requirement that, at the end of the protocol, they can both independently verify that the outcome corresponds to $M_{ij}$ for at least one guess made during the protocol. 
A non-deterministic protocol is called \emph{unambiguous} if for any input $i,j$, exactly one guess allows to verify the value of $M_{ij}$. The \emph{complexity} of a non-deterministic protocol is the maximum (over all inputs and guesses) amount of bits exchanged during the protocol. 
Non-deterministic protocols of complexity $c$ provide a cover of $M$ with at most $2^c$ monochromatic rectangles, which is a partition in the case the protocol is unambiguous. Moreover, each partition of $M$ in $N$ monochromatic rectangles corresponds to an unambiguous protocol of complexity $\lceil\log_2 N\rceil$, where Alice guesses the rectangle covering $i,j$. 

We want to mention another class of communication protocols that is relevant to extended formulations, namely \emph{randomized} protocols that compute a (non-negative) matrix in expectations. These generalize both deterministic and unambiguous non-deterministic protocols and have been defined in \cite{faenza2015extended}, where they are shown to be equivalent to non-negative factorizations (see the next section for the definition of the latter) and to essentially capture the notion of extension complexity. In fact, every extended formulation is obtained from a simple randomized protocol, see again~\cite{faenza2015extended}. 

\medskip

\noindent {\bf Extended formulations and how to find them.} 
We follow here the framework introduced in~\cite{pashkovich2012extended}, that extends~\cite{yannakakis1991expressing}. Consider a pair of polytopes $(P,Q)$ with $P=\conv(v_1,\dots, v_n)$ $\subseteq Q=\{x\in \mathbb{R}^d: Ax\leq b\}\subseteq \R^d$, where $A$ has $m$ rows $a_1,\dots,a_m$.
A polytope $R\in \R^{d'}$ is an \emph{extension} for the pair $(P,Q)$ if there is a projection $\pi: \R^{d'}\rightarrow \R^d$ such that $P\subseteq \pi(R)\subseteq Q$. An \emph{extended formulation} for $(P,Q)$ is a set of linear inequalities describing $R$ as above, and the minimum number of inequalities in an extended formulation for $(P,Q)$ is its \emph{extension complexity}. The \emph{slack matrix} $M(P,Q)$ of the pair $(P,Q)$ is the non-negative $m\times n$ matrix with $M(P,Q)_{i,j}=b_i-a_i^\top v_j$, where $a_i$ is the $i$-th row of $A$. A \emph{non-negative factorization} of $M$ is a pair of non-negative matrices $(T,U)$ such that $M=TU$. The \emph{non-negative rank} of $M$ is the smallest intermediate dimension in a non-negative factorization of $M$. 

\begin{thm}\cite{pashkovich2012extended}  [Yannakakis' Theorem for pairs of polytopes]\label{thm:yan}
	Given a slack matrix $M$ of a pair of polytopes $(P,Q)$ of dimension at least $1$, the extension complexity of $(P,Q)$ is equal to the non-negative rank of $M$. In particular, if $M=TU$ is a non-negative factorization of $M$, then $P\subseteq \{x: \;\exists\; y\geq 0: Ax+Ty=b\}\subseteq Q.$
\end{thm} 

Hence, a factorization of the slack matrix of intermediate dimension $N$ gives an extended formulation of size $N$ (i.e.,~with $N$ inequalities). However such formulation has as many equations as the number of rows of $A$.

Now assume we have a deterministic protocol of complexity $c$ for computing $M=M(P,Q)$. 
The protocol gives a partition of $M$ into at most $2^c$ monochromatic rectangles. This implies that $M=R_1+\dots+R_N$, where $N\leq 2^c$ and each $R_i$ is a rank 1 matrix corresponding to a monochromatic rectangle of non-zero value. Hence $M$ can be written as a product of two non-negative matrices $T, U$ of intermediate dimension $N$, where $T_{i,j}=1$ if the (monochromatic) rectangle $R_j$ contains row index $i$ and 0 otherwise, and $U_{i,j}$ is equal to the value of $R_i$ if $R_i$ contains column index $j$, and 0 otherwise. As a consequence of Theorem \ref{thm:yan}, this yields an extended formulation for $(P,Q)$. In particular, let $\calR^1$ be the set of monochromatic, non-zero rectangles of $M$ produced by the protocol and, for $i=1,\dots,m$, let ${\calR}^1_i\subset {\calR}^1$ be the set of rectangles whose row index set includes $i$. Then the following is an extended formulation for $(P,Q)$:

\vspace{-.1cm}

\begin{align}\label{ef}
a_i x+\sum_{R\in {\calR}^1_i} y_R=b_i & \;\;\;\; \forall \;i=1,\dots, m\\
y\geq 0&\nonumber
\end{align} Again, the formulation has as many equations as the number of rows of $A$, and it is not clear how to get rid of non-redundant equations efficiently. Note that all definitions and facts from this section specialize to those from~\cite{yannakakis1991expressing} for a single polytope when $P=Q$.

\medskip

\noindent {\bf Stable set polytope and $\qstab(G)$.}
The stable set polytope $\stab(G)$ is the convex hull of the characteristic vectors of stable (also, independent) sets of a graph $G$. It has exponential extension complexity \cite{fiorini2012linear,goos2018extension}.  The \emph{clique relaxation} of $\stab(G)$ is:
\begin{equation}\label{eq:perfectgraph}
\qstab(G)=\left\{ x \in \mathbb{R}^d_+ : \sum_{v\in C} x_v \leq 1 \hbox{ for all cliques } C \hbox{ of }G\right\}.
\end{equation}

Note that in~\eqref{eq:perfectgraph} one could restrict to maximal cliques, even though in the following we will consider all cliques when convenient. As a consequence of the equivalence between separation and optimization, optimizing over $\qstab(G)$ is NP-hard  for general graphs, see e.g.~\cite{schrijver2002combinatorial}.
However, the clique relaxation is exact for perfect graphs, for which the optimization problem is polynomial-time solvable using semidefinite programming (see Section~\ref{sec:intro}):
\begin{thm}[\cite{Chvatal75}]\label{thm:stabperfect}
A graph $G$ is perfect if and only if $\stab(G)=\qstab(G)$.
\end{thm}

A fundamental result from \cite{yannakakis1991expressing} is the following.

\begin{thm}\label{thm:protocol-pg}
Let $G$ be a graph with $n$ vertices. There is a deterministic protocol of complexity $O(\log^2n)$ computing the slack matrix of the pair $(STAB(G), QSTAB(G))$. Hence, there is an extended formulation of size $n^{O(\log n)}$ for $(STAB(G), QSTAB(G))$. 
\end{thm}

We remark that, when $G$ is perfect, Theorem \ref{thm:protocol-pg} gives a quasipolynomial size extended formulation for $STAB(G)$. However, as discussed above, the Theorem does not give a subexponential algorithm to obtain such formulation.

\section{A general approach}

 We present here a general technique to explicitly and efficiently produce extended formulations from deterministic protocols, starting with an informal discussion. 

It is important to address the issue of what is our input, and what assumptions we need in order to get an ``efficient" algorithm. 
Recall that a deterministic protocol is identified with a tuple $(\tau,\ell,\Lambda, \{S_v\}_{v \in {\cal V}})$, and the monochromatic rectangles corresponding to leaves of the protocol are denoted by ${\cal R}$. We will assume that $\tau,\ell,\Lambda$ are given to us explicitly, as their size is linearly correlated to the size of the extended formulation we want to produce. However, since in our setting the coefficient matrix $A$ describing $P$ is thought as being exponential in size, the sets $S_v$ have also in general exponential size: hence these are not part of our input.
It is reasonable to assume that we have an implicit description of $A$ and of the $S_v$'s for all leaves of the protocol. In particular, we will require as input certain linear formulations related to $A$ and to the rectangles in $\calR$.

The natural approach to reduce the size of \eqref{ef} is to eliminate redundant equations. However, the structure of the coefficient matrix depends both on $A$ and on rectangles $\calR_i$'s of the factorization, which can have a complex behaviour. The reader is encouraged to try e.g.~on the extended formulations obtained via Yannakakis' protocol for $\stab(G)$, $G$ perfect: the sets $\calR_i$'s have very non-trivial relations with each other that depend heavily on the graph, and we did not manage to directly reduce the system \eqref{ef} for general perfect graphs. Theorem \ref{thm:algdetgeneral} shows how to bypass this problem. 
Informally, we shift the problem of eliminating redundant equations from the system \eqref{ef} to a family of systems $\{A_R x+\mathbf{1}\Lambda_R=b_R\}$,
one for each monochromatic rectangle $R$ of value $\Lambda_R$ produced by the protocol (see Lemma \ref{lem:ef-leaves}). 
 Those systems can still have exponential size, but they may be much easier to deal with since they do not have extra variables. We remark that we need to deal with a system as above for \emph{every} monochromatic rectangle, including those corresponding to output of value $0$. 

\smallskip

We now switch gears and make the discussion formal. Let us start by recalling a well-known theorem from Balas \cite{balas1979disjunctive}, in a version given by Weltge (\cite{weltge2015sizes}, Section 3.1.1).
\begin{thm}\label{thm:balas}
	Let $P_1,P_2\subset \R^d$ be polytopes, with $P_i=\pi_i\{y\in \R^{m_i}: A^iy\leq b^i\}$, where $\pi_i: \R^{m_i}\rightarrow d$ is a linear map, for $i=1,2$. Let $P=\conv(P_1\cup P_2)$. Then we have:
	\begin{align*}\label{eq:balas}
	P=\{x\in \R^d:& \;\exists\; y^1\in \R^{m_1},y^2\in \R^{m_2}, \lambda\in \R:  x=\pi_1(y^1)+\pi_2(y^2),
	\\ 
	&A^1y^1\leq \lambda b^1, A^2y^2\leq (1-\lambda) b^2, 0\leq \lambda\leq 1\}.
	\end{align*}
	Moreover, the inequality $\lambda\geq 0$ ($\lambda\leq 1$ respectively) is redundant if $P_1$ ($P_2$) has dimension at least 1. Hence $\xc(P)\leq \xc(P_1)+\xc(P_2)+|\{i: \dim(P_i)=0\}|$.
\end{thm}

We now give the main theorem of this section. Note that, while the result relies on the existence of a deterministic protocol $(\tau,\ell,\Lambda,\{S_v\}_{v \in {\cal V}})$, its complexity does not depend on the encoding of $\Lambda$ and $\{S_v\}_{v \in {\cal V}}$ (see the discussion above).

\begin{thm}\label{thm:algdetgeneral}
	Let $S$ be a slack matrix for a pair $(P,Q)$, where $P=\conv\{x^*_1,\dots,x^*_n\}\subseteq \R^d$ and $Q=\{x\in \R^d: a_ix\leq b_i $ for $i=1,\dots,m, \; \ell_j \leq x_j \leq u_j \hbox{ for $j \in [d]$}\}$. Assume there exists a deterministic protocol $(\tau,\ell,\Lambda,\{S_v\}_{v \in {\cal V}})$ with  complexity $c$ computing $S$, and let ${\cal R}$ be the set of monochromatic rectangles in which it partitions $S$ (hence $c\leq \lceil \log_2 |\calR|\rceil$). 
	For $R\in \calR$, let $P_R=\conv\{x^*_j: j $ is a column of $R\}$ and $Q_R=\{x\in \R^d: a_ix\leq b_i \;\forall \;i$ row of $R; \ \ell_j\leq x_j\leq u_j \hbox{ for all $j \in [d]$} \}$. 
	
	Suppose we are given $\tau,\ell$ and  for each $R \in {\cal R}$ an extended formulation $T_R$ for $(P_R,Q_R)$. Let $\sigma(T_R)$ be the size (number of inequalities) of $T_R$, and $\sigma_+(T_R)$ be the total encoding length of the description of $T_R$ (including the number of inequalities, variables and equations).
	Then we can construct an extended formulation for $(P,Q)$ of size linear in $\sum_{R\in \calR}\sigma(T_R)$ in time linear in $\sum_{R\in\calR} \sigma_+(T_R)$.
\end{thm}
\begin{proof}

	We can assume without loss of generality that $\tau$ is a complete binary tree, i.e. each node of the protocol other than the leaves has exactly two children. Let ${\cal V}$ be the set of nodes of $\tau$ and $v \in {\cal V}$. Recall that $S_v$ is the (non-necessarily monochromatic) rectangle given by all pairs $(i,j)$ such that, on input $(i,j)$, the execution of the protocol visits node $v$.
	Let us define, for any such $S_v$, a pair $(P_v, Q_v)$ with $P_v=\conv\{x^*_j: j  \hbox{ is a column of } S_v\}$ and $$Q_v=\{x\in \R^d: a_ix\leq b_i \;\forall \;i \hbox{ row of }S_v; \  \ell_j\leq x_j\leq u_j \hbox{ for all $j \in [d]$}\}.$$ Clearly $P_v\subseteq P\subseteq Q \subseteq Q_v$, and $Q_v$ is a polytope. Moreover, $S_\rho=S, P_\rho=P$, and $Q_\rho=Q$ for the root $\rho$ of $\tau$. We now show how to obtain an extended formulation $T_v$ for the pair $(P_v,Q_v)$ given extended formulations $T_{v_i}$'s for $(P_{v_i},Q_{v_i})$, $i=0,1$, where $v_0$ (resp. $v_1$) are the two children nodes of $v$ in $\tau$.
	
	Assume first that $v$ is labelled $A$. Then we have $S_v^T=\bmat{S_{v_0}& S_{v_1}}^T$ (up to permutation of rows), since the bit sent by Alice at $v$ splits $S_v$ in two rectangles by rows -- those corresponding to rows where she sends $1$ and those corresponding to rows where she sends $0$. Therefore $P_v=P_{v_0}=P_{v_1}$ and $Q_v=Q_{v_0}\cap Q_{v_1}$. Hence we have $P_v\subseteq \pi_0 (T_{v_{0}})\cap \pi_1(T_{v_{1}})\subseteq Q_v$, where for $i=1,2$, $\pi_i$ is a projection from the space of $T_{v_i}$ to $\R^d$. An extended formulation for  $T_v:=\pi_0 (T_{v_{0}})\cap \pi_1(T_{v_{1}})$ can be obtained efficiently by juxtaposing the formulations of $T_{v_{0}}, T_{v_{1}}$. 
	
	Now assume that $v$ is labelled $B$. Then similarly we have $S_v=\bmat{S_{v_0}& S_{v_1}}$ (up to permutations of columns). Hence, $P_v=\conv\{P_{v_0}\cup P_{v_1}\}$ and $Q_v=Q_{v_0}=Q_{v_1}$, which implies $P_v\subseteq \conv\{\pi_0(T_{v_{0}})\cup \pi_1(T_{v_{1}})\}\subseteq Q_v$.
	An extended formulation for  $T_v:=\conv\{\pi_0(T_{v_{0}})\cup \pi_1(T_{v_{1}})\}$ can be obtained efficiently by applying Theorem \ref{thm:balas} to the formulations of $T_{v_{0}}, T_{v_{1}}$. 
	Iterating this procedure, in a bottom-up approach we can obtain an extended formulation for $(P,Q)$ from extended formulations of $(P_v,Q_v)$, for each leaf $v$ of the protocol.
	
	We now bound the encoding size of formulation. If $T_v = \pi_0 (T_{v_{0}})\cap \pi_1(T_{v_{1}})$, then $\sigma_+(T_v)\leq \sigma_+(T_{v_{0}})+\sigma_+(T_{v_{1}})$. Consider now $T_v=\conv\{\pi_0(T_{v_{0}})\cup \pi_1(T_{v_{1}})\}$. From Theorem \ref{thm:balas} we have $\sigma_+(T_v)\leq \sigma_+(T_{v_{0}})+\sigma_+(T_{v_{1}})+O(d)$. Since the binary tree associated to the protocol is complete, it has size linear in the number of leaves, 
	hence for the final formulation $T_\rho$ we have 
	
	\vspace{-.25cm}
	
	$$\sigma_+(T_\rho)\leq\sum_{R\in \calR} (\sigma_+(T_R)+O(d))=O\left(\sum_{R\in \calR} \sigma_+(T_R)\right),$$
	 where the last equation is justified by the fact that we can assume $\sigma_+(T_R)\geq d$ for any $R \in {\cal R}$. The bounds on the size of $T_\rho$ and on the time needed to construct the formulation are derived analogously. \end{proof}

A couple of remarks on Theorem~\ref{thm:algdetgeneral} are in order. The reader may recognize similarities between the proof of Theorem~\ref{thm:algdetgeneral} and that of the main result in~\cite{fiorini1711strengthening}, where a technique is given to construct approximate extended formulations for polytopes using Boolean formulas. While similar in flavour, those two results seem incomparable, in the sense that one does not follow from the other. 
They both fall under a more general framework, in which one derives extended formulations for a polytope by associating a tree to it, starting from simpler formulations and taking intersection and convex hulls following the structure of the tree. However, we are not aware of any other applications of this framework, hence we do not discuss it here.

The formulation that is produced by Theorem~\ref{thm:algdetgeneral} may not have \emph{exactly} the form given by the corresponding protocol. Also, even for the special case $P=Q$, the proof relies on the version of Yannakakis' theorem for pairs of polytopes. On the other hand, it does not strictly require that we reach the leaves of the protocol -- a similar bottom-up approach would work starting at any node $v$, as long as we have an  extended formulation for $(P_v,Q_v)$. However, if we indeed reach the leaves, there is an extended formulation for each $(P_R,Q_R)$ that has a very special structure, as next lemma shows. 

\begin{lem}\label{lem:ef-leaves}
Using the notation from Theorem~\ref{thm:algdetgeneral}, let $R \in {\cal R}$ and denote by $\Lambda_R\geq 0$ its value. Then we have that $P_R\subseteq T_R^* \subseteq Q_R$, where
\begin{equation}\label{eq:ef-from-leaves}
T_R^*:=\{ A_R x + \mathbf{1}{\Lambda_R } = b_R, \, \ell_j \leq x_j \leq u_j \hbox{ for all $j \in [d]$}\},\end{equation}
$\mathbf{1}$ is the all-1 vector of appropriate length, and $A_R$ (resp. $b_R$) is the submatrix (resp. subvector) of the constraint matrix  (resp. of the right-hand side) describing $Q$ corresponding to rows of $R$.
\end{lem}

\begin{proof}
By construction, every vertex of $P_R$ satisfies $A_R x + \mathbf{1}\Lambda_R = b_R$, hence $P_R\subseteq T_R^*$.
Moreover, since $\Lambda_R\geq 0$, $A_R x \leq b_R$ is clearly valid for $T_R^*$, showing $T_R^* \subseteq Q_R$. 
\end{proof}

We remark that, in case the value of a monochromatic rectangle $R$ is not known, one can just replace in~\eqref{eq:ef-from-leaves} $\Lambda_R$ with a variable $y_R$, together with the constraint $y_R\geq 0$: similarly as in the proof above, one checks that the resulting $T_R^*$ is an extended formulation of $(P_R,Q_R)$.

\subsection{Applications} \label{subsec:applications}

\noindent {$(STAB(G),QSTAB(G))$.} We now describe how to apply Theorem \ref{thm:algdetgeneral} to the protocol from Theorem \ref{thm:protocol-pg} as to obtain an extended formulation for $(STAB(G),QSTAB(G))$ in time $n^{O(\log n)}$. In particular, this gives an extended formulation for $STAB(G)$, $G$ perfect within the same time bound.

      We first give a modified version of the protocol from \cite{yannakakis1991expressing}, stressing a few details that will be important in the following. The reader familiar with the original protocol can immediately verify its correctness.
Let $v_1,\dots,v_n$ be the vertices of $G$ in any order. At the beginning of the protocol, Alice is given a clique $C$ of $G$ as input and Bob a stable set $S$, and they want to compute the entry of the slack matrix of $\stab(G)$ corresponding to $C,S$, i.e. to establish whether $C, S$ intersect or not.

At each stage of the protocol, the vertices of the current graph $G=(V,E)$ are partitioned between \emph{low degree} $L$ (i.e., at most $|V|/2$) and \emph{high degree} $H$. Suppose first $|L|\geq |V|/2$. Alice sends (i) the index of the low degree vertex of smallest index in $C$, or (ii) $0$ if no such vertex exists. In case (i), if $v_i\in S$, then $C\cap S\neq \emptyset$ and the protocol ends; else, $G$ is replaced by $G\cap N(v_i)\setminus\{v_j \in L : j <i\}$, where $G\cap U$ denotes the subgraph of $G$ induced by $U$. In case (ii), if Bob has no high degree vertex, then $C \cap S= \emptyset$ and the protocol ends, else, $G$ is replaced by $G\cap H$. If conversely $|L|< |V|/2$, then the protocol proceeds symmetrically to above: Bob sends (i) the index of the high degree vertex of smallest index in $S$, or (ii) $0$ if no such vertex exists. In case (i), if $v_i\in C$, then $C\cap S\neq \emptyset$ and the protocol ends; else, $G$ is replaced by $G\cap \bar N(v_i)\setminus\{v_j \in H : j <i \}$. In case (ii), if Alice has no low degree vertex, then $C \cap S= \emptyset$ and the protocol ends, else, $G$ is replaced by $G\cap L$. Note that  at each step the number of vertices of the graph is decreased by at least half, and $C$ and $S$ do not intersect in any of the vertices that have been removed.

Now let $S$ be the slack matrix of the pair $(STAB(G),QSTAB(G))$. Each monochromatic rectangle $R\in \calR$ in which the protocol from Theorem \ref{thm:protocol-pg} partitions $S$ is univocally  identified by the list of cliques and of stable sets corresponding to its rows and columns. With a slight abuse of notation, for a clique $C$ (resp. stable set $S$) whose corresponding row is in $R$, we write $C\in R$ (resp. $S \in R$), and we also write $(C,S)\in R$. 
We let $P_R$ be the convex hull of stable sets $S \in R$ and $Q_R$ the set of clique inequalities corresponding to cliques $C \in R$, together with the unit cube constraints. 

We need a fact on the structure of $\calR$, for which we introduce some more notation: for a (monochromatic) rectangle $R\in \calR$, let $C_R$ be the set of vertices sent by Alice and $S_R$  the set of vertices sent by Bob during the corresponding execution of the protocol. Note that $C_R$ is a clique and $S_R$ is a stable set. 

\begin{obs}\label{eq:obs-CrSr}
For each $R\in \calR$, there is exactly one clique $C$ and one stable set $S$ of $G$ such that $C=C_R$ and $S=S_R$. Conversely, given a clique $C$ and a stable set $S$, there is at most one rectangle $R\in\calR$ such that $C=C_R$ and $S=S_R$. Notice that $|C_R|+|S_R|\leq \logn$ for any $R\in \calR$.
\end{obs}

Recall that, to apply Theorem \ref{thm:algdetgeneral} we need to have a description of $\tau,\ell$ and the extended formulations for $(P_R,Q_R)$, for each $R\in \calR$. These are computed as follows.
\begin{enumerate}
\item \emph{$\tau$, $\ell$, and $(C_R,S_R)$ for all $R \in {\cal R}$.}  Enumerate all cliques and stable sets of $G$ of combined size at most $\logn$ and run the protocol on each of those input pairs. Each of those inputs gives a path in the tree (with the corresponding $\ell$), terminating in a leaf $v$, corresponding to a rectangle $R$.
By Observation \ref{eq:obs-CrSr}, $\tau$ is given by the union of those paths. Moreover, observe that, for each $R \in {\cal R}$, $C_R$ (resp. $S_R$) is contained in all cliques $C$ (resp. stable sets $S$) such that, on input $(C,S)$, the protocol terminates in the leaf $v$ corresponding to $R$. In particular, on input $(C_R,S_R)$, the protocol terminates in $v$. Hence, the inclusion-wise minimal such $C$ and $S$ give $C_R,S_R$.  

\item \emph{For each leaf of $\tau$ corresponding to a rectangle $R\in {\cal R}$, give a compact extended formulation $T_R$ for the pair $(P_R,Q_R)$.}  
We follow an approach inspired by Lemma \ref{lem:ef-leaves}. We first need the following fact on the structure of the rectangles.

\begin{lem}\label{prop:rectangles}
	Let $R=(C_R,S_R)\in \calR$ and $(C,S)\in R$.
	Then $(C',S')\in R$ for any $C'$ such that $C_R\subseteq C' \subseteq C$ and $S'$ such that $S_R\subseteq S' \subseteq S$.
\end{lem}
\begin{proof}
Note that a vertex $v\in C\setminus C_R$ is not sent during the protocol on input $(C,S)$. Hence, the execution of the protocols on inputs $(C,S)$ and $(C-v,S)$ coincides. Indeed at every step Alice chooses the first vertex of low degree in her current clique, and if $v$ is never chosen, having $v$ in the clique does not affect her choice. Moreover, the choice of Bob only depends on his current stable set and the vertices previously sent by Alice. In particular, we have 
$(C\setminus \{v\},S) \in R$. Iterating the argument (and applying the symmetric for $v \in S\setminus S_R$) we conclude the proof.
\end{proof}

Now, let $R$ be a rectangle of value 1. We claim that $P_R\subseteq T_R \subseteq Q_R$, where
\begin{align*}
T_R=\{x\in \R^d: \;
& x_v= 0 \;  \;\forall\; v\in C_R \mbox{ or } v\in V\setminus C_R: C_R+v\in R,\\
&  0 \leq x \leq 1\}.\end{align*}
The first inclusion is due to Lemma \ref{prop:rectangles} and to the fact that, since $R$ has value 1, all vertices of $P_R$ correspond to stable sets that are disjoint from cliques of $R$. For the second inclusion, given a clique $C$ of $R$ and $v\in C$, using again Lemma \ref{prop:rectangles} we have that $C_R+v$ is also a clique of $R$, so for $x\in T_R$ we have $x(C)=\sum_{v\in C} x_v=0\leq 1$. Observe that, in order to decide if $C_R + v \in R$, it suffices to run the protocol on input $(C_R+v, S_R)$.

Finally, let $R$ be a rectangle of value $0$. We have $C_R\cap S_R=\{u\}$ for some vertex $u$. One can conclude, in a similar way as before, that $P_R\subseteq T_R \subseteq Q_R$, where
\begin{align*}
T_R=\{x\in \R^d: \;
& x_u=1,\\
& x_v= 0 \;  \;\forall\; v\in C_R-u \mbox{ or } v\in V\setminus C_R: C_R+v\in R\\
& 0 \leq x \leq 1\}.\end{align*}
\end{enumerate}

We conclude by observing that the approach described above proceeds by obtaining the leaves of $\tau$ through an enumeration of all cliques and stable sets of combined size ${\logn}$, and then reconstructing $\tau$. This takes time $n^{\Theta(\logn)}$. However, one could instead try to construct $\tau$ from the root, by distinguishing cases for each possible bit sent by Alice or Bob. This intuition is the basis for the alternative formulation that we give in Section~\ref{sec:direct}.

\medskip

\noindent {\bf Min-up/min-down polytopes.} We give here another application of Theorem~\ref{thm:algdetgeneral}. Min-up/min-down polytopes were introduced in \cite{lee2004min} to model scheduling problems with machines that have a physical constraint on the frequency of switches between the operating and not operating states. For a vector $x\in\{0,1\}^T$ and an index $i$ with $1\leq i\leq T-1$, let us call $i$ a \emph{switch-on index} if $x_i=0$ and $x_{i+1}=1$, a \emph{switch-off index} if $x_i=1$ and $x_{i+1}=0$. 
  For $L,\ell, T \in \N$ with $\ell\leq T, L\leq T$, the min-up/min-down polytope $P_T(L,\ell)$ is defined as the convex hull of vectors in $\{0,1\}^T$ satisfying the following: for any $i,j$ switch indices, with $i<j$, we have $j-i\geq \ell$ if $i$ is a switch-on index and $j$ a switch-off index and $j-i\geq L$ if viceversa $i$ is a switch-off index and $j$ a switch-on index. In other words, in these vectors (seen as strings) each block of consecutive zeros (resp. ones) has length at least $L$ (resp. $\ell$). In \cite{lee2004min}, the following is shown:
\begin{thm}\label{thm:minupdown}
The following is a complete, non-redundant description of $P_T(L,\ell)$:
\begin{align*}
   \sum_{j=1}^k (-1)^{j}x_{i_j}  \leq 0\;\; & \;\forall\; \{i_1,\dots,i_k\}\subseteq [T]: k \mbox { odd and } i_k-i_1\leq L \\
    \sum_{j=1}^k (-1)^{j-1}x_{i_j} \leq 1 \;& \;\forall\; \{i_1,\dots,i_k\}\subseteq [T]: k \mbox { odd and } i_k-i_1\leq \ell.   
\end{align*}
\end{thm}
 When $\ell,L$ are not constant, $P_T(L,\ell)$ has an exponential number of facets. An extended formulation for $P_T(L,\ell)$ is given in \cite{rajan2005minimum}.
We now provide a simple deterministic protocol for its slack matrix $S^{\ell,L,T}$. As a simple consequence of Theorem \ref{thm:algdetgeneral}, this will give a compact extended formulation that can be obtained efficiently.  For simplicity we only consider the part of the slack matrix indexed by the first set of inequalities, as the protocol for the second part can be derived in an analogous way. Hence, assume that Alice is given an index set $I=\{i_1,\dots,i_k\}\subseteq [T]$ with $k$ odd and $i_k-i_1\leq L$, and Bob is given a vertex $v$ of $P_T(L,\ell)$, which is univocally determined by its switch indices. Alice sends to Bob the index $i_1$. Assume that $v_{i_1}=0$, the other case being analogous. Then $v$ can have at most one switch-on index and at most one switch-off index in $\{i_1,\dots,i_1+L\}$, hence in particular in $I$. Bob sends to Alice 1 bit to signal $v_{i_1}=0$, and the coordinates of such indices. Hence, Alice now knows exactly the coordinates of $v$ on $I$ and can output the slack corresponding to $I,v$. In total, it is easy to see that the protocol has complexity $\lceil \log T \rceil + 2\max(\lceil\log L\rceil, \lceil\log \ell\rceil))$, and can be modeled by a tree $\tau$ of size $O(T\cdot (L+\ell)^2)$. 

Finally, in order to apply Theorem \ref{thm:algdetgeneral} we would need to obtain a compact extended formulation for the pair $(P_R,Q_R)$ for each rectangle $R$ corresponding to a leaf of $\tau$. While this can be achieved by applying Lemma \ref{lem:ef-leaves} with a tedious case distinction, we proceed in a simpler, alternative way. We exploit the fact that, as remarked above, when applying Theorem \ref{thm:algdetgeneral}, we do not necessarily need to give formulations to the leaves $\tau$, but we can start from their ancestors as well.
Consider a step of the protocol, corresponding to a node $v$ of $\tau$, where the following communication has taken place: Alice sent an index $i_1$, Bob sent a bit to signal that $v_{i_1}=1$ and the switch-off index $i^*$ (the other cases are dealt with similarly). This determines a rectangle $S_v=R$ that is not monochromatic (as its value depends on Alice's input) and has as columns all the vertices $x$ of $P_T(L,\ell)$ with 
\begin{equation}\label{eq:JonLee}x_{i_1}=\dots=x_{i^*}=1, \; x_{i^*+1}=\dots=x_{i_1+L}=0\end{equation}
and as rows all the inequalities corresponding to subsets $I$ whose smallest index is $i_1$. Consider the polytope
$$ T_R :=\{ 0 \leq x \leq 1 : \, x \hbox{ satisfies}~\eqref{eq:JonLee}\}.$$

Clearly, we have $P_{R}\subseteq T_R \subseteq Q_{R}$. Now, it is easy to see that a similar formulation can be given to all nodes $v$ of $\tau$ whose children are leaves. Hence we can conclude the following.

\begin{thm}\label{thm:minupdownEF}
Let $L,\ell, T$ be positive integers with $\ell\leq T, L\leq T$. The min-up/min-down polytope $P_T(L,\ell)$ has an extended formulation of size $O(T\cdot (L+\ell)^2)$ that can be written down in time $O(T\cdot (L+\ell)^2)$.
\end{thm}

\section{Direct derivations}\label{sec:direct} 
In this section, we show how to construct certain extended formulations directly from protocols, without resorting to Theorem~\ref{thm:algdetgeneral}.

\medskip

\noindent {\bf Complement graphs.} Denote by $\bar G$ the complement of a graph $G$. An extended formulation for $(\stab(\bar G),\qstab(\bar G))$ can be efficiently obtained from an extended formulation of $(\stab(G),\qstab(G))$, keeping a similar dimension (including the number of equations). 

We begin by observing that the following holds, for any graph $G$: $\stab(G)\subseteq \{x: x\geq 0, x^Ty\leq 1 \;\forall\; y\in \stab(\bar{G})\}= \qstab(G)$. We will prove a stronger statement in Lemma \ref{lem:complementpolar}, and in order to efficiently turn this into an extended formulation, we need the following fact. 
\begin{lem}[\cite{Martin91,weltge2015sizes}] \label{lem:polar} Given a non-empty polyhedron $Q$ and $\gamma\in \R$, let $P=\{x: x^Ty\leq \gamma \; \forall \; y\in Q\}$. If $Q=\{y: \exists \;z : Ay+Bz\leq b, Cy+Dz=d\}$, then we have that
$$
P=\{x: \exists \lambda\geq 0, \ \mu: A^T\lambda+C^T\mu=x, \quad B^T\lambda +D^T\mu=0, \ b^T\lambda+d^T\mu\leq \gamma\}.
$$
In particular, $xc(P)\leq xc(Q)+1$.
\end{lem}

\begin{lem} \label{lem:complementpolar}
Let $G$ be a graph on $n$ vertices such that $(\stab(\bar{G}), \qstab( \bar G))$ admits an extended formulation $Q$ with $r$ additional variables (i.e. $n+r$ variables in total), $m$ inequalities and $k$ equations. Then $(\stab(G),\qstab(G))$ admits an extended formulation with $m+k$ additional variables, $m+1$ inequalities, $n+r$ equations, which can be written down efficiently given $Q$.
\end{lem}

\begin{proof}
We have $\stab(\bar{G})\subseteq \pi(Q) \subseteq \qstab(\bar{G})$, where $\pi$ is the projection on the original space. Let $\eta(Q)$ denote the extended formulation for $\{x: x\geq 0, x^Ty\leq 1 \;\forall\; y\in \pi(Q)\}$ obtained applying Lemma~\ref{lem:polar} to $Q$ and then adding non-negativity constraints. Writing down $\eta(Q)$ can be done in linear time in the encoding size of $Q$, and the number of variables, inequalities and equations of $\pi(Q)$ can be derived immediately from Lemma \ref{lem:polar}. We now prove that $\eta(Q)$ is an extended formulation for $(\stab(G),\qstab(G))$.
    Let $S$ be a stable set in $G$, and $\chi^S$ be the corresponding incidence vector. For any $y\in \pi(Q)$, we have $y^T\chi^S = y(S)\leq 1$ as $S$ is a clique in $\bar{G}$, hence since $\eta(Q)$ (hence $\pi(\eta(Q))$) is convex it follows that $\stab(G)\subseteq \pi(\eta(Q))$. Now, for a clique $C$ of $G$ and $x\in \pi(\eta(Q))$, one has $\chi^C\in \pi(Q)$ hence $x(C)=x^T\chi^C\leq 1$, proving $\pi(\eta(Q))\subseteq \qstab(G)$. 
\end{proof}

To conclude, we remark that the above statement produces an exact formulation of $\stab(G)$ whenever $G$ is a perfect graph.

\medskip

\noindent {\bf Alternative extended formulation for $(\stab(G),\qstab(G))$.} We now present a general algorithmic framework to obtain extended formulations of $(\stab(G),$ $\qstab(G))$ by iteratively decomposing a graph $G$. In particular this results in an algorithm that, given a graph $G$ on $n$ vertices, produces an explicit extended formulation for $(\stab(G),\qstab(G))$ of size $n^{O(\log n)}$, in time bounded by $n^{O(\log n)}$. The decomposition approach is inspired by  Yannakakis' protocol \cite{yannakakis1991expressing},  even though the formulation obtained has a different form than~\eqref{ef} (and different from the formulation obtained in the previous section). We first recall a simple, though important observation:

\begin{obs}\label{obs:decompose}
Let $P_1,\dots, P_k\in \R^n$ be polyhedra with $P=P_1\cap\dots\cap P_k$, and let $Q_i$ be an extended formulation for $P_i$ for $i=1,\dots,k$, i.e. $P_i=\{x\in \R^n: \;\exists \;y^{(i)}\in \R^{r_i} :(x,y^{(i)})\in Q_i\}$. Then $P=\{x\in \R^n: \mbox{ for } i=1,\dots,k \; \exists \; y^{(i)}\in \R^{r_i}: (x,y^{(i)})\in Q_i \}$.
\end{obs}

A key tool is next Lemma. For a vertex $v$ of $G$, $N^+(v):=N(v)\cup \{v\}$ denotes the inclusive neighbourhood of $v$.

\begin{lem}\label{lem:stabdecompose}
Let $G$ be a graph on vertex set $V=\{v_1,\dots,v_n\}$, and, fix $k$ with $1\leq k \leq n$. Let $G_i$ be the induced subgraph of $G$ on vertex set $V_i=N^+(v_i)\setminus\{v_1,\dots,v_{i-1}\}$ for $i=1,\dots,k$, and $G_0$ the induced subgraph of $G$ on vertex set $V_0=\{v_{k+1},\dots, v_n\}$. Let $Q_i$ be an extended formulation for $(\stab(G_i), \qstab(G_i))$, for $i=0,\dots,k$, and
$P=\{x\in \R^V: x_{V_i} \in \pi_i(Q_i)\mbox{ for } i=0,\dots, k \}$, where $x_{V_i}$ denotes the restriction of $x$ to the coordinates in $V_i$, and $\pi_i$ is the projection from the space of $Q_i$ to $R^{V_i}$. Then $\stab(G)\subseteq P \subseteq \qstab(G)$. In particular, the formulation obtained by juxtaposing the descriptions of the $Q_i$'s as in Observation \ref{obs:decompose} is an extended formulation for $(\stab(G),\qstab(G))$. 
\end{lem}
\begin{proof}
 For the first inclusion, let $\chi^S$ be the characteristic vector of a stable set $S$ of $G$. Then, for every $i=0,\dots,k$,  $S\cap V(G_i)$ is a stable set in $G_i$, hence $\chi^{S}_{V_i}\in \pi_i(Q_i)$. By convexity, we conclude that $\stab(G)\subseteq P$. For the second inclusion, let $x\in P$, and let $C$ be a clique of $G$.
    Notice that $C$ is a clique of $G_i$ for some $i$: indeed, if $\{v_1,\dots,v_k\}\cap C\neq \emptyset$, let $i$ be the minimum such that $v_i\in C$, then $C\subseteq N^+(v_i)\setminus \{v_1,\dots,v_{i-1}\}=V_i$ by definition; if $\{v_1,\dots,v_k\}\cap C= \emptyset$, then $C\subseteq V_0$. Since $\pi_i(Q_i)\subseteq \qstab(G_i)$, we have $x(C)=x_{V_i}(C)\leq 1$.
\end{proof}

 Given a graph $G$, a \emph{decomposition tree} $\tau_G$ is a rooted tree defined inductively as follows:
\begin{itemize}
    \item The root node of $\tau_G$ is associated to the graph $G$;
    \item Let $v$ be a node of $\tau_G$, to which graph $H$ is associated. Then $v$ is either a leaf, or it has a single child, to which graph $\bar H$ is associated, or it has children $\tau_{H_0},\dots,\tau_{H_k}$, for some $k\geq 1$, associated to graphs $H_0, H_1, \dots, H_k$, defined as in Lemma \ref{lem:stabdecompose} for some $v_1,\dots,v_k$, respectively.
\end{itemize}

We now state the main lemma of the section.

\begin{lem}\label{lem:rooted-tree-to-ef}
Let $G$ be a graph on $n$ vertices and let $\tau$ be a decomposition tree associated to $G$. Denote by ${\cal L}_\tau$ the set of leaves of $\tau$. Suppose that we are given, for each leaf $L\in {\cal L}_\tau$, an extended formulation $\eta(G_L)$ for $(\stab(G_L),\qstab(G_L))$, where $G_L$ is the graph associated to node $L$. Let $\sigma$ be an upper bound on the total encoding length of the description of any $\eta(G_L)$ (including the number of inequalities, variables and equations), and denote by $|\tau|$ the number of vertices of $\tau$.  Then we can construct an extended formulation for $(\stab(G),\qstab(G))$ of size $O( |\tau| \sigma)$ in time $O( |\tau| \sigma)$.
\end{lem}

\begin{proof} We will abuse notation and identify a node of the decomposition tree and the corresponding subgraph. Consider the extended formulation, which we call $\eta(G)$, obtained by traversing $\tau$ bottom-up and applying the following: 
\begin{enumerate}
    \item if a non-leaf node $H$ of $\tau$ has a single child $\bar{H}$, then define $\eta(H)$ to be equal to the extended formulation of $\{x\in \mathbb{R}^{V(H)}: x^Ty\leq 1 \;\forall\; y \in \eta(\bar{H})\}$, obtained by applying Lemma \ref{lem:polar}, plus constraints $x \geq 0$.
    \item otherwise if $H$ has children $H_0,\dots,H_k$, then define $\eta(H)$ to be the extended formulation obtained from $\eta(H_0),\dots, \eta(H_k)$ using Observation~\ref{obs:decompose}. 
\end{enumerate}  

We first prove that $\eta(G)$ is an extended formulation for $(\stab(G),\qstab(G))$. We proceed by induction on the height of $\tau$, in particular we prove that for a node $H$ of $\tau$, $\eta(H)$ is an extended formulation of $(\stab(H),\qstab(H))$, assuming this is true for the children of $H$. If $H$ is a leaf of $\tau$, then there is nothing to prove. Otherwise, we need to analyze two cases: \begin{enumerate}
    \item $H$ has a single child, labelled $\bar{H}$. Then, assuming that $\stab(\bar{H})\subseteq \pi(\eta(\bar{H})) \subseteq \qstab(\bar{H})$, we have that $\stab(H)\subseteq \pi(\eta(H))\subseteq \qstab(H)$ by Lemma \ref{lem:complementpolar}.
    \item $H$ has children $H_0,\dots,H_k$. Assume that for $i=0,\dots,k$, $\stab(H_i)\subseteq \eta(H_i)\subseteq \qstab(H_i)$. Then, applying Lemma \ref{lem:stabdecompose}, we have that $\stab(H)\subseteq \pi(\eta(H))\subseteq \qstab(H)$.
\end{enumerate}

We now deal with complexity issues. Thanks to Lemma \ref{lem:polar} and Observation \ref{obs:decompose}, the extended formulation $\eta(H)$ for a node $H$ of $\tau$ has linear size in the total size of the formulations associated to the children of $H$, and the same bound holds for the time needed to write down $\eta(H)$. Iterating this over all the nodes of $\tau$, from the leaves to the root, we obtain the desired bounds on the size of $\eta(G)$ and the time needed to write it down.
\end{proof}

Lemma~\ref{lem:rooted-tree-to-ef} allows us to deduce the following.

 \begin{thm}\label{thm:generalgraphs}
 Let $G$ be a graph on $n$ vertices. Then there is an algorithm that, on input $G$, outputs an extended formulation of $(\stab(G),\qstab(G))$ of size $n^{O(\log n)}$ in $n^{O(\log n)}$ time.
 \end{thm}
 
 \begin{proof} Define the decomposition tree ${\tau_G}$ iteratively as follows, starting from the root. Let $H$ be the current graph. If $H$ has less than $c$ vertices, where $c$ is any fixed constant, then $H$ has no child. Otherwise, let $v_1,\dots, v_k$ be the vertices of $H$ of degree at most $|V(H)|/2$, if $k\geq |V(H)|/2$, then $H$ has children labelled ${H_0},\dots, {H_k}$, where $H_0, \dots, H_k$ are the graphs defined as in Lemma \ref{lem:stabdecompose}.  Notice that $H_0,\dots, H_k$ have at most $|V(H)|/2+1$ vertices. Otherwise, $H$ has a single child labelled ${\bar{H}}$. Now it is easy to see that $\tau_G$ has $n^{O(\log n)}$ nodes. Indeed, notice that, for any $H$, at least half of the vertices of either $H$ or its complement $\bar{H}$ will have low degree, hence the corresponding node will be branched and the size of the graphs associated to its children will reduce by roughly half: this implies that the height of $\tau_G$ is $O(\log n)$, and the fact that a node of $\tau_G$ can have at most $n$ children implies that the total number of vertices is $n^{O(\log n)}$. $\tau_G$ is a decomposition tree for $G$, and we can associate to each of its leaves labelled by a (constant size) subgraph $G_L$ the clique formulation $QSTAB(G_L)$, which has constant size. Applying Lemma~\ref{lem:rooted-tree-to-ef} concludes the proof.
 \end{proof}
 
  Theorem \ref{thm:generalgraphs} may have interesting computational consequences: indeed, there is interest in producing relaxations of $\stab(G)$ without explicitly computing the Theta body (see for instance \cite{giandomenico2015ellipsoidal}). We conclude by remarking that, as a consequence of the main result of~\cite{goos2015lower}, there exists a constant $c>0$ such that, there is no extended formulation of size $O(n^{\log^c n})$ for $(\stab(G),\qstab(G))$: this limits the extent to which Theorem~\ref{thm:generalgraphs} could be improved for general (non-perfect) graphs.

\medskip

\noindent {\bf Claw-free graphs and generalizations.} Let $P=\stab(G)$, where $G$ is claw-free, i.e., it does not contain the complete bipartite graph $K_{1,3}$ as induced subgraph. Extended formulations for this polytope are known~\cite{faenza2012separating}, but as the stable set polytope of claw-free graphs generalizes the matching polytope, it follows from~\cite{rothvoss2017matching} that it has no extended formulation of polynomial size. The row submatrix of the slack matrix of $P$ corresponding to clique constraints can however be computed by the following simple protocol from~\cite{faenza2015extended}. Fix an arbitrary order $v_1,\dots,v_k$ of $V$. Alice, who has a clique $C$ as input, sends vertex $v_i\in C$ with smallest $i$ to Bob, who has a stable set $S$. If $v \in S$, Bob outputs $0$. Else, since $G$ is claw-free, we have $|N(v)\cap S|\leq 2$, and clearly $C\subset N(v)$, hence Bob sends $N(v)\cap S$ and Alice can compute and output $1-|C\cap S|$. The protocol has complexity at most $3\log n +1$, hence by applying Theorem \ref{thm:yan} we get the following formulation of size $O(n^3)$:
 \begin{align}\label{ef:stabclawfree}
x(C)+\sum_{R\in \calR_C} y_R=1 & \;\;\;\; \forall \;C \mbox{ clique of } G\\
x, y\geq 0&\nonumber
\end{align}
where, following the notation from~\eqref{ef}, $\calR^1$ contains a $1$-rectangle for each couple $(v,U)$, where $v\in V$ and $U\subseteq N(v)$ with $U$ stable (i.e., $|U|\leq 2$), and, for a clique $C$, $\calR_C$ denotes all rectangles $(v_i,U)$ from ${\calR^1}$ with $i = \arg\min\{j \in [n] : v_j \in C\}$ and $C\cap U=\emptyset$. 

We now remove from~\eqref{ef:stabclawfree} many redundant equations. Before, we notice that the above protocol can be easily generalized to $K_{1,t}$-free graphs for $t\geq 3$: in this case sets $\calR^1, \calR_C$ are defined similarly as before, except that now we have rectangles $(v,U)$ with $|U|\leq t-1$. This yields a formulation of size $O(n^t)$. We state our result for this more general class of graphs: informally, the only clique equations that we keep are coming from singletons and edges, obtaining a formulation with only $O(n^2)$ many equations.

\begin{thm}\label{thm:clawfree}
Let $G(V,E)$ be a  $K_{1,t}$-free graph. Let ${\calR^1}, \calR_C$ as above. Then the following is an extended formulation for $(\stab(G),\qstab(G))$:
\begin{align}\label{ef:clawfree}
x(v)+\sum_{R\in \calR_v} y_R=1 & \;\;\;\; \forall \;v \in V\\
x(e)+\sum_{R\in \calR_e} y_R=1 & \;\;\;\; \forall \;e \in E\nonumber\\
x,y\geq 0,&\nonumber
\end{align} where we abbreviated $\calR_v = \calR_{\{v\}}$. In particular, if $G$ is also perfect,~\eqref{ef:clawfree} is a formulation for $\stab(G)$.
\end{thm}
\begin{proof}
 Thanks to the above discussion, we only need to show that, for any clique $C$ of $G$ with $|C|=k\geq 3$, the equation $x(C)+\sum_{R\in \calR_C} y_R=1$ is implied by the equations in~\eqref{ef:clawfree}. In fact, since equations from~\eqref{ef:clawfree} are valid for~\eqref{ef:stabclawfree} (they are a subset of valid equations), it will be enough to prove that the \emph{left-hand side} of any equation from~\eqref{ef:stabclawfree} is a linear combination of the left-hand sides of equations from~\eqref{ef:clawfree}. From now on, fix such $C$, let $v\in C$ be the first vertex of $C$ (in the order fixed by the protocol) and consider the following expression, obtained by summing the left-hand side of the equations relative to $e=uv$, for every $u\in C-v$:
$$
\sum_{\substack{e=uv:\\ u\in C-v}} \left(x(u)+x(v)+\sum_{R\in \calR_e} y_R\right)=
$$
$$
(k-2)x(v)+x(C)+ \sum_{\substack{e=uv:\\ u\in C-v}} \left(\sum_{R\in \calR_e\cap\calR_C} y_R+\sum_{R\in \calR_e\setminus\calR_C} y_R\right)
$$
Now, consider a rectangle $R=(v,U)\in \calR^1$, and let $(C,S)$ be an entry of $R$ (hence, $C\cap S=\emptyset$). Then $v$ is the first vertex of $C$ and $U= N(v)\cap S$, with $U\cap C=\emptyset$. Hence, we can derive $\calR_e=\{(v,U): U\subset N(v), U \in\calS, u\not\in U\}$ for $u \in C$ and $e=uv$, and  $\calR_C=\{(v,U): U\subset N(v), U\in\calS,  U\cap C=\emptyset\}$, where $\calS$ denotes the family of the stable sets of $G$. Hence $\calR_C\subseteq \calR_e$ for $e\subset C$. We can then rewrite the above expression as:
 \begin{align}\label{clawfreeeq}
 (k-2)x(v)+x(C)+ (k-1) \sum_{\substack{U\subseteq N(v),\\ U\in\calS,  U\cap C=\emptyset}} y_{(v, U)}+\sum_{u\in C\setminus\{v\}} \sum_{\substack{U\subseteq N(v)-u\\ U\in\calS, U\cap C\neq \emptyset }} y_{{(v, U)}}= \nonumber\\
 (k-2)x(v)+x(C)+ (k-1) \sum_{\substack{ U\subseteq N(v),\\ U\in\calS,  U\cap C=\emptyset}} y_{(v, U)} +(k-2) \sum_{\substack{U\subseteq N(v)\\ U\in\calS, U\cap C\neq \emptyset }} y_{(v, U)},
\end{align} {where we used that $U\in\calS, U\cap C\neq \emptyset$ implies that $|U\cap C| = 1$, hence $y_{(v, U)}$ will appear in all summations, except the one corresponding to $\{u\} = U\cap C$.}

Now, consider the {left-hand} side of the equation corresponding to ${C=\{v\}}$: 
$$
x(v)+\sum_{R\in \calR_v} y_R=x(v)+\sum_{U\subset N(v),U\in \calS} y_{{(v,U)}}. $$
Subtracting $k-2$ times the latter from \eqref{clawfreeeq} we obtain as required $$x(C)+\sum_{U\subset N(v),\\ U\in\calS,  U\cap C=\emptyset}y_{v,U} = x(C) + \sum_{R\in \calR_C} y_R.$$\end{proof}

\noindent {\bf Comparability graphs. }
A graph $G$ on vertex set $D$ is a \emph{comparability} graph if there is a partial order $(D,\leq_D)$ such that two elements of $D$ are adjacent in $G$ if and only if they are comparable with respect to $\leq_D$. A clique (resp. stable set) in $G$ corresponds to a chain (resp. antichain) in $(D,\leq_D)$. It is well known that comparability graphs are perfect. In \cite{yannakakis1991expressing}, it is described an unambiguous nondeterministic protocol for the slack matrix of $\stab(G)$, which we now recall. Given a clique $C=\{v_1,\dots,v_k\}$ with $v_1\leq\dots\leq v_k$ in $D$, and a stable set $S$ disjoint from $C$, there are three cases: 1) every node of $C$ precedes some node of $S$ (equivalently, $v_k$ does); 2) no node of $C$ precedes a node of $S$ (equivalently, $v_1$ does not precede any node of $S$); 3) there is an $i$ such that $v_i$ precedes some node of $i$, and $v_{i+1}$ does not.
Alice, given $C$, guesses which of the three cases applies and sends to Bob the certificate {$(v_k, 0)$} in case 1), {$(v_1, 1)$} in case 2) and $(v_i,v_j)$ in case 3).  This protocol yields a factorization of the slack matrix, hence an extended formulation for $\stab(G)$ of the usual kind:
\begin{align}\label{ef:comp}
x(C)+ y(v_1,{1})+y(v_1,v_2)+\dots +y(v_k,{0})=1 & \;\;\;\; \forall \;C=\{v_1,\dots,v_k\} \in G\\ 
x,y\geq 0&\nonumber
\end{align}

\begin{lem}
Let $G(V,E)$ be a comparability graph with order $\leq_D$, then the following is an extended formulation for $\stab(G)$:
\begin{align}\label{ef:compsmall}
x(v)+y(v,{0})+y(v,{1})=1  &\;\;\;\; \forall \; v \in V\\
x(u)+x(v)+y(u,{1})+y(u,v)+y(v,{0})=1  &\;\;\;\; \forall \; u,v\in V: u\leq_D v\nonumber\\
x,y\geq 0&\nonumber.
\end{align}
\end{lem}
\begin{proof}
{Constraints~\eqref{ef:compsmall} are a subset of~\eqref{ef:comp}, hence their projection $T$ to the space of $x$ variables satisfy $\stab(G)\subseteq T$}. Let $(x,y)$ be a point that satisfies \eqref{ef:compsmall}, and $C=\{v_1,\dots,v_k\}$ a clique of $G$ with $v_1\leq_D\dots\leq_D v_k$, $k\geq 3$. {We show $x(C)\leq 1$.} Manipulating the equations from \eqref{ef:compsmall}, we have that for $i=2,\dots, k$, $x(v_i)=y(v_{i-1},0)-y(v_{i-1},v_i) -y(v_i,{0})$. Hence:
$$
x(C)=
$$
$$
x(v_1)+y(v_{1},0)-y(v_{1},v_2) -y(v_2,0)+\dots+y(v_{k-1},0)-y(v_{k-1},v_k) -y(v_k,0)
$$
$$
=x(v_1)+y(v_{1},0)-y(v_{1},v_2)-\dots -y(v_{k-1},v_k) -y(v_k,0)\leq 1,
$$ as required.
\end{proof}

We remark that a small, explicit extended formulation for the stable set polytopes of comparability graphs has been given in \cite{lovasz1994stable}.

\medskip

\noindent{\bf Threshold-free graphs.}
A threshold graph is a graph for which there is an ordering of the vertices $v_1,\dots, v_n$, such that for each $i$ $v_i$ is either complete or anticomplete to $v_{i+1},\dots,v_n$. Fix a threshold graph $H$ on $t$ vertices, where $t$ is a constant. We say that a graph is $H$-free if it does not contain $H$ as an induced subgraph. A deterministic protocol for the clique-stable set incidence matrix of $H$-free graphs is known~\cite{lagoutte}. It implies a polynomial size extended formulation for {$(\stab(G),QSTAB(G))$} if $G$ is $H$-free. One can apply Theorem \ref{thm:algdetgeneral} to show that such formulation can be obtained efficiently. One can also obtain this result directly {by applying Lemma~\ref{lem:rooted-tree-to-ef}.}

\begin{thm}\label{thm:threshold}
Let $H$ be a fixed threshold graph, and let $G$ be an $H$-free graph on $n$ vertices. Then there is an algorithm that, on input $G$, outputs an extended formulation for $(\stab(G),\qstab(G))$ of size $O(n^{|V(H)|})$ in time $O(n^{|V(H)|})$. If moreover $G$ is also perfect, then the output is an extended formulation of $\stab(G)$.
\end{thm}
\begin{proof}
We will use Lemma~\ref{lem:rooted-tree-to-ef}, hence we will describe a decomposition tree $\tau_G$ of size $O(n^{|V(H)|})$, whose leaves are singletons.
Let $V(H)=\{u_1,\dots,u_t\}$ with $u_i$ either complete or anticomplete to $u_{i+1},\dots,u_t$, for $i=1,\dots,t-1$.
The root of $\tau_G$ is $G$ (as before we identify a node of $\tau_G$ with a graph). If $u_1$ is complete to the other vertices of $H$, then $G$ has $n$ children $G_1,\dots,G_n$. Otherwise, $G$ has one child $\bar{G}$, which in turn has $n$ children $\bar{G}_1,\dots,\bar{G}_n$.
We now recurse on the children obtained (except for those which are already labelled by singletons, which are left as leaves): each $G_i$ (or $\bar{G}_i$) has either $|V(G_i)|$ children if $u_2$ is complete to $u_3,\dots,u_t$, or it has one child $\bar{G_i}$ which in turn has $|V(G_i)|$ children, and so on for $u_3,\dots,u_{t-1}
$. 
Our tree has clearly size $O(n^t)$. We now show that all its leaves are (labelled by) singletons, which concludes the proof. Let $G'$ be a leaf and consider the sequence $G^{(0)}=G,G^{(1)},\dots, G^{(k)}=G'$ of graphs that represent labels in the path from the root of $\tau_G$ to $G'$, excluding graphs that are complement of their predecessor. If $k\leq t-2$, then the decomposition must have stopped at $G'$ because $G'$ is a singleton. Assume by contradiction that $k=t-1$, and let $v_1,\dots, v_{t-1}$ be the vertices of $G$ such that $G^{(i)}=G^{(i-1)}_{v_i}$ or $G^{(i)}=\bar{G}^{(i-1)}_{v_i}$ (according to $u_i$) for $i=1,\dots,t-1$, and let $v_t$ be a vertex of $G^{(t)}$. We have that for $i=2,\dots, t$, $v_i\in N(v_{i-1})$ or $v_i\in V(G)\setminus N(v_{i-1})$, again according to $u_{i}$, but this implies that $v_1,\dots, v_{t}$ induce a subgraph of $G$ isomorphic to $H$, a contradiction.
\end{proof}

\section{Extended formulations from non-deterministic, unambiguous protocols} {As mentioned in the introduction, extended formulations can be obtained also via \emph{unambiguous non-deterministic} and \emph{randomized} communication protocols. The latter are very general, and it is not clear how to extend Theorem~\ref{thm:algdetgeneral} to deal with them. In this section, we deal therefore with the former.}

Let $P\subseteq Q$ be a pair of polytopes and $M$ the slack matrix of $(P,Q)$. Assume that we are given a partition of $M$ in monochromatic rectangles $\calR=\{R_1,\dots,R_t\}$ (similarly as in Theorem \ref{thm:algdetgeneral} this assumption could be relaxed). This partition corresponds to a non-deterministic unambiguous protocol for $M$ of complexity $\log(t) + O(1)$. On one hand, this implies a non-negative factorization of $M$, hence an extended formulation of $(P,Q)$, of size $O(t)$, but it is not clear at all how to get this formulation in general. On the other hand, Yannakakis~\cite{yannakakis1991expressing} showed how to deduce from $\calR$ a deterministic protocol for $M$ of complexity $O(\log^2 t)$, hence an extended formulation for $P$ of size $t^{O(\log t)}$. Unfortunately, Theorem~\ref{thm:algdetgeneral} cannot directly be applied to produce this formulation.  However, a formulation of size $t^{O(\log t)}$ can be obtained in time $t^{O(\log t)}$ under some basic additional assumptions. 

Let us begin by describing a version of Yannakakis' reduction mentioned above that can be found in~\cite{raobook}. In this setting, Alice has as input a row index $r$ of $M$, Bob a column index $c$, and the goal of the protocol is to determine the unique rectangle of $R_{r,c} \in \calR$ containing the entry $(r,c)$. We say that a rectangle $R$ is \emph{horizontally good} if $r$ is a row of $R$, and $R$  intersects horizontally at most half of the rectangles in $\calR$, \emph{vertically good} if $c$ is a column of $R$ and $R$ intersects vertically at most half of the rectangles in $\calR$, and \emph{good} if it is horizontally or vertically good. Notice that, since two rectangles in $\calR$ cannot intersect both vertically and horizontally, $R_{r,c}$ is good. We now proceed in stages: in each stage, Alice, which has input $r$, sends a horizontally good rectangle $R$, or the information that there is none: in the former case, either $c$ is a column of $R$, in which case $R=R_{r,c}$ and the protocol ends, or Alice and Bob can delete from $\calR$ all the rectangles that do not horizontally intersect $R$, as they cannot be $R_{r,c}$; in the latter case, there must be a vertically good rectangle $R$, which Bob sends to Alice, and again either $R=R_{r,c}$ or Alice and Bob can delete from $\calR$ all the rectangles that do not vertically intersect $R$. At the end of the stage, the size of $\calR$ is decreased by at least half. Note that $R_{r,c}$ is never deleted from ${\cal R}$ during the protocol, hence it will eventually be sent by Alice or Bob. Hence, there will be at most $\lceil\log t \rceil$ stages before $R_{r,c}$ is sent, and the protocol has complexity $O(\log^2 t)$. Therefore, the protocol partitions $M$ in $t^{O(\log t)}$ monochromatic rectangles, each of which is contained in a unique rectangle of $\calR$, which is the one that is output at the end of the protocol.

We now show how to obtain a tree modeling this protocol, and how to construct an extended formulation for $P,Q$ from the tree. We assume that we are given two graphs $G_H, G_V$ on vertex set $\calR$, where two rectangles are adjacent in $G_H$ if they intersect horizontally, and in $G_V$ if they intersect vertically.

\begin{thm}
Let $P,Q,M,\calR$ as above. Assume that we are given the graphs $G_H, G_V$, and an extended formulation $T_R$ of the pair $(P_R,Q_R)$ (defined exactly as in Theorem \ref{thm:algdetgeneral}), for each $R\in\calR$. 
Let $\sigma$ be an upper bound on the size of any $T_R$, and $\sigma_+$ on its total encoding length. Then there is an algorithm that, on input $\{T_R: R\in \calR\}$ and $\calR$, outputs an extended formulation for $(P,Q)$ of size $t^{O(\log (t))}$poly$(\sigma)$ in time $t^{O(\log (t))}$poly$(\sigma_+)$.
\end{thm}
\begin{proof}
In Section \ref{sec:prelim} we formally defined a protocol as a tuple ($\tau,\ell,\Gamma,\{S_v\}_{v\in V(\tau)})$. Here we will use a slightly different notation to define and efficiently construct a tree $\tau$. For ease of exposition the tree will not be binary. We also associate to each node $v$ of $\tau$ a subset $\calR_v\subseteq\calR$. $\tau$ is defined in an iterative manner as follows. The root $\rho$ of $\tau$ is an Alice node and $\calR_\tau = \calR$. We attach to $\rho$ a child $v$ for each rectangle $R_v$ that can be possibly sent by Alice during this first step, i.e., each rectangle with degree at most $|\calR|/2$ in $G_H$. Each such child $v$ is an Alice node, we define $\calR_v$ removing from $\calR$ rectangles that are not $R_v$ or its neighbors in $G_H$. We also add another children $v$ of $\rho$ for the information that Alice cannot send any rectangle. In this case, $\calR_v$ is defined to be the subset of rectangles corresponding to nodes of $G_H$ of degree more than $|\calR|/2$. 

If $u$ is a generic Alice node with $|\calR_u|\geq 2$, we similarly attach to $u$ an edge to an Alice child $v$ for each rectangle $R_v$ that, in the subgraph $G_H[\calR_u]$ (i.e., induced by $\calR_u$), has low degree (at most $|\calR_u|/2$), and define $\calR_v$ being the inclusive neighborhood of $R$ in $G_H[\calR_u]$. We also attached another children $v$ for the information that Alice cannot send any rectangle, and define ${\cal R}_v$ as above.

Similarly, for a Bob node $u$ with $|\calR_u|\geq 2$, there is an Alice node $v$ for each rectangle $R_v$ possibly sent by Bob (there must be at least one by construction), and define $\calR_v$ to be the subset of $\calR_u$ of rectangles corresponding to nodes of high-degree in $G_H[\calR_v]$ (there must be at least one by construction).  A node $v$ is a leaf if $\calR_v=\{R\}$ for some $R\in\calR$, and we set $\Gamma(v)=R$.

$\tau$ as constructed above has size $t^{O(\log t)}$ and models the above protocol \emph{for any possible matrix $M$} whose partition in rectangles $\calR$ respects the structure of $G_H, G_V$. In other words, for each such matrix $M$ and for each input $r,c$, there is a root-to-leaf path in $\tau$ that corresponds to the execution of the protocol for $M$ on input $r,c$, and leads to the unique rectangle of $\calR$ containing $r,c$. On the other hand, depending on $M$ some of the nodes of $\tau$ might never be traversed during an execution: for instance, if in the first step of the protocol the low degree rectangles that Alice can send cover all the rows of $M$, then for every input row Alice will have a horizontally good rectangle to send, and the Bob child of the root will never be traversed. 

\smallskip

Consider $\tau$ as above and to each leaf $v$ with $\Gamma(v)=R$ we associate the extended formulation $T_R$ such that $P_R\subseteq \pi(T_R)\subseteq Q_R$, where $P_R,Q_R, \pi $ are defined as in Theorem~\ref{thm:algdetgeneral}. We proceed bottom up similarly to the proof of Theorem~\ref{thm:algdetgeneral} to obtain formulations $T_v$ for each $v$ node of $\tau$: if $v$ is an Alice node, $T_v$ is the intersection of $T_{v_i}$ for all the children $v_i$ of $v$, and if it is a Bob node, $T_v$ is obtained as their convex hull (using Theorem \ref{thm:balas}). We claim that $T_\rho$ is an extended formulation for $(P,Q)$.

To argue this, we will proceed by contradiction. First, we introduce for each node $v$ of $\tau$ a set $A_v$ of row indices of $M$ and $B_v$ of column indices of $M$ defined iteratively as follows. For the root $\rho$ of $\tau$, $A_\rho, B_\rho$ are the row set and the column set of $M$. For an Alice node $v$, with children $v_1,\dots, v_k$, we let $B_{v_i}=B_v$ for any $i$, and define $A_{v_1},\dots A_{v_k}$ to be a partition of $A_v$ defined as follows: $r \in A_v$ is in $A_{v_i}$ if, for a column $c$ of $B_v$, on input $r,c$ the execution of the protocol traverses node $v_i$. Notice that this partitions $A_v$ as the message sent by Alice at step $v$ can only depend on her input and on the knowledge that $c\in B_v$, hence for a given row, the choice of $c$ is irrelevant (i.e., requiring that the node is traversed for some $c\in B_v$ or for all of them is the same). Also, some of the $A_{v_i}$ can be empty. We proceed similarly for a Bob node $v$ with children $v_1,\dots, v_k$, letting $A_{v_i}=A_v$ for any $i$, and $B_{v_1},\dots,B_{v_k}$ be the partition of $B_v$ corresponding to the information that Bob sends (where some $B_{v_i}$ can be empty).

Now, assume by contradiction that $T_\rho$ is not an extended formulation for $(P,Q)$. First, assume that $P\not\subset \pi(T_\rho)$, in particular that there is vertex $x^*$ of $P$, corresponding to a column $c^*$ of $M$, with $x^*\not\in \pi(T_\rho)$. Notice that $c^*\in B_\rho$. We color red some vertices of $\tau$ according to the following rules: $\rho$ is red; for a red node $v$ and a child $u$ of $v$, if $x^*\not\in \pi(T_{u})$, color $u$ red. Notice that a red Alice node will have at least a red child, and all children of a red Bob node will be red. Now, color blue some vertices of $\tau$ according to the following rules: $\rho$ is blue; for a blue node $v$ and a child of $u$ with $c^*\in B_u$, color $u$ blue. Notice that a blue Bob node will have at least a blue child, and all children of a blue Alice node will be blue. Now, this implies that there is a path from $\rho$ to a leaf $v$ whose nodes are both red and blue. Let $R=\Gamma(v)$. Since $v$ is blue, we have that $c^*\in B_v$, hence $c^*\in R$, but since $v$ is red, $x^*\not\in \pi(T_v)=\pi(T_R)$, a contradiction to the fact that $T_R$ is an extended formulation of $(P_R,Q_R)$. Hence we proved $P\subseteq \pi(T_\rho)$.  Proceeding in an analogous way one proves that $\pi(T_\rho)\subseteq Q$, concluding the proof.\end{proof}

\bigskip

\noindent {\bf Acknowledgements.} We thank Mihalis Yannakakis for inspiring discussions and Samuel Fiorini for useful comments on \cite{aprile2018thesis}, where many of the results here presented appeared. Manuel Aprile would also like to thank Aur\'{e}lie Lagoutte and Nicolas Bousquet for useful discussions.

\printbibliography

\end{document}